\documentclass[5p]{elsarticle}

\usepackage{graphicx}
\usepackage{amssymb}

\usepackage{amssymb,amsfonts}
\usepackage[all,arc]{xy}
\usepackage{enumerate}
\usepackage{mathrsfs}
\usepackage{natbib}
\usepackage{color}
\usepackage[right]{showlabels}
\usepackage{verbatim}
\usepackage{amsthm}
\newtheorem{thm}{Theorem}[section]

\newtheorem{lem}[thm]{Lemma}
\newtheorem{asm}[thm]{Assumption}

\newdefinition{defn}[thm]{Definition}
\newdefinition{defns}[thm]{Definitions}

\newcommand{\ud}{\mathrm{d}}

\newcommand{\cx}[1]{{\color{blue}Xialiang: {#1}}}

\def\P{{\mathbb P}}
\def\Q{{\mathbb Q}}
\def\E{{\mathbb E}}

\usepackage[english]{babel}
\usepackage[utf8x]{inputenc}
\usepackage[T1]{fontenc}

\usepackage{graphicx} 
\usepackage[colorlinks=False]{hyperref} 
\usepackage{amsmath}  
\usepackage{amsfonts} %
\usepackage{amssymb}  %

\usepackage{subcaption}

\usepackage{bm}
\journal{Operations Research Letters}

\begin{document}

\begin{frontmatter}


\title{Distributionally Robust Optimization with Correlated Data from Vector Autoregressive Processes }


\author[UC]{Xialiang Dou} \author[UC,ANL]{Mihai Anitescu}

\address[UC]{University of Chicago, Department of Statistics}
\address[ANL]{Argonne National Laboratory, Mathematics and Computer Science Division}
\begin{abstract}

We present a distributionally robust formulation of a stochastic optimization problem for non-i.i.d vector autoregressive data. We use the Wasserstein distance to define robustness in the space of distributions and we show, using duality theory, that the problem is equivalent to a finite convex-concave saddle point problem. The performance of the method is demonstrated on both synthetic and real data.
\end{abstract}


\end{frontmatter}


\section{Introduction}
A common formulation of optimization under uncertainty is the following stochastic program: \cite{shapiro2009lectures}
$$ \min_{x\in\mathcal{D}}  \mathbb{E}_{\mathbf{y} \sim F} [h(x,\mathbf{y})].$$
 Here the decision variable $x$ has a convex feasible domain $\mathcal{D}$, $h$ is a convex function in $x$, and $y$ is data from an underlying  generating process. 
Much of the work in stochastic programming is carried out under the assumption that the distribution $F$ is known 
\cite{birge2011introduction,shapiro2009lectures}.
 In many problems, however, other than some basic properties, we do not have the exact description of $F$.  
Using the empirical distribution $F_n$ as a surrogate for $F$ would overfit the data, especially when we  have  very few samples. 
 
One way to overcome the uncertainty attached to the probability density $F$ itself is to investigate \textit{distributionally robust stochastic optimization} (DRSO). This problem is \[\min_{x \in \mathcal{D}} \max_{F \in \mathcal{U}} \mathbb{E}_{\mathbf{y} \sim F} [h(x, \mathbf{y})],\] where the distribution $F$ is from a set $\mathcal{U}$. Significant research  recently has been carried out concerning the choice of ambiguity set $\mathcal{U}$ 
by trying to balance out-of-sample performance and computational complexity. In \cite{delage2010distributionally}, the authors proposed to specify the ambiguity set by the first and one-sided second moment constraints in order to preserve the convexity of the formulation. As mentioned in \cite{gao2017distributionally}, however, the one-sided second moment constraint may have no effect on the problem.
Other approaches identify the ambiguity set by considering distributions that are close to the empirical distribution in an appropriate measure. Different metrics include the Kullback-Leibler divergence \cite{jiang2016data}, Burg entropy \cite{wang2016likelihood}, total variation \cite{sun2015convergence}, $\chi^2$-distance \cite{klabjan2013robust}, and more generally $\phi$-divergence \cite{ben2013robust}  \cite{bayraksan2015data}. A drawback of $\phi$-divergence formulations is that they may not be rich enough to capture distributions of interest \cite{gao2016distributionally}. 
Recent work has introduced DRSO formulations based on the Wasserstein distance \cite{esfahani2015data, gao2016distributionally}, which has both out-of-sample performance guarantees and computational efficient reformulations.

	Most of the cited references study the problem in a setting where the data consists of copies of independent and identically distributed (i.i.d) random variables. In many applications, however, particularly when the data is formed by sequential entries in a time series, the i.i.d. assumption is not realistic. In this work we study the case when we have times series data from a \textit{vector autoregression} (VAR) process
	\begin{equation}\label{eq:vartime}
	\mathbf{y}_{t+1} = A\mathbf{y}_t + \bm{\xi}_t \text{, }t =1,\dots,n-1 .
	\end{equation}
	Here $\xi_1,\dots,\xi_{n-1} \in \mathbb{R}^d$ are i.i.d random variables with a zero-mean residual. Realizations $y_1,\dots,y_n \in \mathbb{R}^d$ are our observations, and the conditional expectation satisfies \\ $\mathbb{E}_{t-1}[\mathbf{y}_t] =A\mathbf{y}_{t-1}$. For a more concise way to represent the model, let $Y_+ = [y_2,\dots,y_n] \in \mathbb{R}^{d \times (n-1)}$, $Y_- = [y_1,\dots,y_{n-1}] \in \mathbb{R}^{d \times (n-1)}$ and $E = [\xi_1,\dots,\xi_{n-1}] \in \mathbb{R}^{d \times (n-1)}$. Thus, we have 
	\begin{equation}\label{eq:varmatrix}
	Y_+ = AY_- +E.
	\end{equation}
	The VAR model is widespread. It occurs in econometrics \cite{sims1980macroeconomics}, control theory \cite{kumar2015stochastic}, and  recent brain image analysis \cite{friston2009causal}. The ubiquity of times series models  motivates us to generalize the DRSO to the VAR-dependent data setting.  
	
	Our contribution is to propose a DRSO formulation using Wasserstein distance techniques for VAR data and to prove that if the original problem has the needed convexity features, then our DRSO formulation is a finite-dimensional convex-concave saddle point problem.
		
\section{Model and Robust Formulation}
Suppose we have data $y_1,y_2,\cdots,y_n \in \mathbb{R}^d$ from a time series and we need to make a decision that will be affected by the next outcome $y_{n+1}$. We assume that the time series is generated by a vector autoregression process $\left( \mathrm{VAR}(1) \right)$, 
\[\mathbf{y}_{t+1} = \beta + A\mathbf{y}_{t} + \bm{\xi}_{t}.\]
Here $A$ is a fixed transition matrix, and the noise terms $\ \xi_t\in \Xi \subseteq \mathbb{R}^d, \text{ }t=1 ,\dots, n-1$ are i.i.d with zero mean. For notational simplicity, let $\tilde{\mathbf{y}}_t = [1,\mathbf{y}_t]^T$, $\tilde {\bm{\xi}}_t = [0,\bm{\xi}_t]^T$, and \[\tilde{A} = \left[ \begin{array}{ll} 1 & 0 \\ \beta & A \end{array} \right].\] Therefore, the data model becomes
\[\tilde{\mathbf{y}}_{t+1} =  \tilde{A}\tilde{\mathbf{y}}_{t} + \tilde{ \bm{\xi}}_t, \; t=1,2,\ldots, n-1. \] 
To simplify notation, we will use, without loss of generality, $\mathbf{y}_t$ and $A$ in the previous equation for the rest of the article;  in other words, 
we will refer to the algebraic formalism of \eqref{eq:vartime} and \eqref{eq:varmatrix}. 
Consider the stochastic programming problem
	\begin{eqnarray} \label{eq:SO}
	\begin{aligned}
		& \underset{x}{\text{min }}
		& & \E_{\mathbf{y}_{n+1} \sim F} [ h(x,\mathbf{y}_{n+1}) ] \\
		& \text{subject to} && x \in \mathcal{D}.	
	\end{aligned}    
	\end{eqnarray}
	Here $F$ is the true conditional distribution of $y_{t+1}$ given $y_t$ within the model. For problems with real data, both $A$ and $F$ need to be estimated from the data. We can build confidence intervals of $A$ and $F$ under common regularity assumptions about the noise term $\xi_t$, which lead to our robust formulation. We consider the DRSO problem with decision variable $x$ informed by incoming data from process \eqref{eq:vartime}:
\begin{eqnarray}\label{prob1}
	\begin{aligned}
& \underset{x}{\text{min }} \underset{A,F}{\text{max}}
& & \mathbb{E}_{\mathbf{\bm{\xi}}_{n} \sim F}[ h(x,Ay_n + \bm{\xi}_n) ] \\
& \text{subject to}
& & d_1(A,\hat A) \leq \varepsilon_2\\
&&& F \in \mathcal{U} \\
&&& x \in \mathcal{D}.	
	\end{aligned}    
\end{eqnarray}
Here $\hat A$ is a fixed matrix obtained by regression based on the matrix formulation 
\eqref{eq:varmatrix}:
\[ \hat A = \arg \min_A \|Y_+ - AY_- \|_F^2 .\]
Here $\| \cdot \|_F$ is the Frobenius norm of the matrix. 
Concerning the structural matrix $A$, we impose an estimation accuracy constraint on $\hat A$, whereby $A$ and $\hat A$ have to be relatively close. It is well known that the accuracy of the regression matrix $\hat A$ depends on the condition number of the design matrix $Y_-$ \cite{lai1982least,golub2012matrix}.
Later we will specify the choice of $\varepsilon_2$ such that, with high probability, the true matrix $A$ satisfies that constraint. For each choice of $A$, we get the residual $\hat \xi_i = y_{i+1} - Ay_i$, $i = 1, \dots, n-1$. Let $F_n$ be the empirical distribution of $\hat \xi_i$; that is, $F_n = \frac{1}{n-1} \sum_{t=1}^{n-1} \delta_{\hat \xi_i}(\xi)$. 
Note that $F_n$ depends on $A$, which is itself a variable in \eqref{prob1}, but for simplicity of notation we do not explicitly indicate that. 
The family of the distribution, $\mathcal{U}$, is specified by constraining the  distribution $F$ relative to the empirical distribution $F_n$ by means of a specially chosen distance function
\begin{eqnarray}\label{cons}
	\mathcal{U} = \left\{
		\begin{array}{l|l}
		 F &F(\mathbf{\xi} \in \Xi) = 1\\
		 &  d_w(F,F_n) \leq \varepsilon_1
		\end{array}
	\right\}.
\end{eqnarray}
The ambiguity set $\mathcal{U}$ is a subset of the distributions on the measurable space $(\mathbb{R}^{d},\mathcal{B})$, where $\mathcal{B}$ is the $\sigma$-algebra of the Borel sets. The first condition constrains the support of the distribution to a known set $\Xi$, and the second constraint regulates the behavior of the  noise term. In the following, we will assume this set to be bounded. The existence of a known set that contains the support of the distribution is a common assumption with other approaches \cite{delage2010distributionally,gao2016distributionally}, at least when aiming for results comparable to ours, as well as a reasonable approach for most physical and economical processes. 


\subsection*{Wasserstein Distance}
The quantity $d_w$ is the Wasserstein distance, which  can be defined as follows.
\begin{defn}\label{def1}
Let $\P, \Q$ be two distributions on a metric space $(X,d)$. The 2-Wasserstein distance can be defined by
\[ d_w(\P,\Q) = \inf_{\pi \in \Pi(\P,\Q)} \sqrt{ \int_{X\times X} d^2(x , x') \pi(\ud x,\ud x')}.\]
 Here $\Pi(\P,\Q)$ is the family of distributions on $X\times X$ with marginal distributions $\P$ and $\Q$ \cite{villani2008optimal}. 
\end{defn}
This definition can be viewed as finding an optimal transport between two distributions, while the cost of moving probability mass is encoded by the distance $d(x,x')$ on the metric space $X$. Although this definition appears daunting, the key observation, also used in \cite{esfahani2015data}, is that when $\Q$ is the empirical distribution $F_{n} = \frac{1}{n-1}\sum_{i=1}^{n-1} \delta_{\hat \xi_i}(\xi)$, we can compute $d_w(\cdot,\cdot)$ relatively easily since we can always break $\P$ down to the sum of $n-1$ conditional distributions $\P_i$. Subsequently, by utilizing duality, we will convert the resulting infinite-dimensional optimization problem \eqref{prob1} into a computable finite convex problem.

\section{Problem Formulation and Dual Representation}
We now formally state our DRSO version of \eqref{prob1}:
	\begin{eqnarray} 
	\begin{aligned} \label{eq:DRO}
		& \underset{x}{\text{min }} \underset{A,F}{\text{max}}
		& & \E_{\mathbf{\bm{\xi}}_{n} \sim F}h(x,\mathbf{y}_{n+1}) \\
		& \text{subject to}
		& & d^2_w(F,F_n) \leq \varepsilon_1 \\
		&&& A \in \Omega(\epsilon_2)\\
		&&& x \in \mathcal{D},
	\end{aligned}    
	\end{eqnarray}
	where $F_n = \frac{1}{n-1}\sum_{i=1}^{n-1} \delta_{\tilde \xi_i}(\xi)$ is the empirical distribution of $\tilde\xi_i = y_{i+1}-Ay_{i}$, $i = 1 ,\dots n-1$, and $\Omega$ defines the uncertainty set of the matrix $A$, 
	\begin{equation} \label{eq:Omega}
	\Omega(\epsilon_2) \stackrel{.}{=}\left\{A \in \mathbb{R}^{d \times d} \left| \left\| A_{i} - \hat A_{i}\right\| \leq \varepsilon_{2,i} \text{ , for $i \in [d]$}  \right. \right\}.	
	\end{equation}
	The second constraint in \eqref{eq:DRO} is the confidence interval of $A$ for which we can choose $\varepsilon_{2,i}$  based on regression analysis \cite{fox1997applied} (see also the end of \S \ref{radius}). 
	Specifically, we denote here and in the following by $A_i$ the $i$th column of matrix $A \in \mathbb{R}^{d \times d}$ and by $\hat{A}_i$ the $i$th column of  matrix $\hat{A}$.  
	\subsection*{Reformulation}
	Writing now expectations in integral form and recalling our specification of the choice of support $\xi$ in \eqref{cons} and of the objects in Definition 
	\ref{def1}, we have the following.
	\begin{equation} 
		\begin{aligned}
			& \underset{x \in \mathcal{D}}{\text{min }}~~\underset{A \in \Omega(\epsilon_2),F,\pi \in \Pi(F,F_n)}{\text{max }}
			& & \int_{\Xi}h(x,Ay_n+\xi)\,  \ud F(  \xi) \\
			& \text{subject to}	
			& & \int_{\Xi} \, \ud F(\xi) = 1\\
			&&& \int_{\Xi \times \Xi} \|\xi - \xi'\|^2 \ud \pi( \xi, \xi') \leq \varepsilon_1\\
		\end{aligned}
	\end{equation} 
	Here the second constraint is a rewrite of the Wasserstein distance constraint using Definition \ref{def1}. From the definition of $\Pi(F,F_n)$, the joint distribution $\pi \in \Pi(F,F_n)$ has marginal distributions $F$ and $F_n$. Since $F_n$ is the empirical distribution, by the rules of conditional  distributions we have that $\pi(\xi,\xi') = \frac{1}{n-1}\sum_{i=1}^{n-1} \pi(\xi|\xi'=\hat \xi_i) \delta_{\hat \xi_i}(\xi')$, where  $F_i\doteq \pi(\xi|\xi'=\hat \xi_i)$  is the conditional distribution of $\xi$ given that $\xi'$ takes the value $\hat \xi_i$. 
	Then we have $F = \sum_{i=1}^{n-1}\P(\xi' = \hat \xi_i)F_i = \frac{1}{n-1}\sum_{i=1}^{n-1}F_i$.  
	We note that, as a conditional distribution, $F_i$ is constrained at this stage only by having the same support as $F$. 
	Also note that since $\pi(F,F_n)$ can be used to define $F$ (as one of its marginals), we 
substitute $F$ as above and reformulate the optimization problem with the conditional probabilities as the variables (similar to
\cite{esfahani2015data} ).
We obtain
	 \begin{equation} \label{integral_formulation}
		\begin{aligned}
			& \underset{x \in \mathcal{D}}{\text{min }}~\underset{A \in \Omega(\epsilon_2),\left\{F_i\right\}_{i=1}^{i=n-1}}{\text{max }}
			& & \frac{1}{n-1}\sum_{i=1}^{n-1}\int_{\Xi}h(x,Ay_n+\xi)\, \ud F_i(\xi) \\
			& \text{subject to}	
			& & \int_{\Xi} \, \ud F_i(\xi) = 1,\; i=1,2,\ldots,n-1 \\
			&&& \frac{1}{n-1}\sum_{i=1}^{n-1} \int_{\Xi} \|\xi - \hat{\xi}_i' \|^2 \ud F_i(\xi) \leq \varepsilon_1.\\
		\end{aligned}
	\end{equation} 
	Now, we reduce the above distributional optimization problem into a finite-dimensional problem.
\begin{thm}	\label{t:main}
Let $\Phi(x,A)$ denote the solution of the inner maximizing problem with fixed $x$ and $A$ in \eqref{integral_formulation}. When $h(x,y)$ is differentiable, convex in the first argument and concave in the second argument,  we have the following identity: 
	 \begin{equation*} 
		\begin{aligned}
			 \Phi(x,A)=\underset{u \geq 0}{\min }  & ~u\varepsilon_1   + \underset{\xi_i \in \Xi, i \in [n-1]}{\max } \Bigg\{ \frac{1}{n-1}\sum_{i=1}^{n-1} \Big[ h(x,Ay_n + \xi_i) -  \\
			 &u\cdot\ \|\xi_i - (y_{i+1}-Ay_i)\|^2 \Big]  \Bigg\}.\\
		\end{aligned}
	\end{equation*} 
\end{thm}
	\begin{proof}
	From Lagrangian duality,  we get that 	$\Phi(x,A)$ equals
		\begin{equation*}
			\begin{aligned}		
		   \max_{F_i, i \in [n-1]} \inf_{u \geq 0} u\left[ \varepsilon_1 -\frac{1}{n-1}\sum_{i=1}^{n-1} \int_{\Xi} \|\xi_i - \tilde\xi'_i\|^2 F_i( \ud  \xi) \right] \\
			\frac{1}{n-1}+\sum_{i=1}^{n-1}\int_{\Xi} h(x, Ay_n+\xi) \ud F_i(\xi)
			\end{aligned}
		\end{equation*}
		\begin{eqnarray*}
			&=&  \inf_{u\geq 0} u\varepsilon_1 + \ \max_{ F_i, i \in [n-1]} \Bigg\{ \frac{1}{n-1}\sum_{i=1}^{n-1} \int_{\Xi} h(x,Ay_n+\xi_i)\\
			&&  -u\| \xi_i - \tilde \xi_i \|^2 F_i(\ud \xi_i)\Bigg\} 
		\end{eqnarray*}
	\begin{eqnarray*}
			&=&  \inf_{u\geq 0} u\varepsilon_1  + \ \max_{\xi_i \in \Xi, i \in [n-1]} \Bigg\{ \frac{1}{n-1}\sum_{i=1}^{n-1} \left[ h(x,Ay_n+\xi_i) \right. \\
			&&  \left. - u\|\xi_i - (y_{i+1} - Ay_i)\|^2 \right] \Bigg\}.
		\end{eqnarray*}
		
	The second equality occurs from exchanging min and max, which is valid by strong duality. This can be proved by an extended version of a well-known strong duality result for moment problems \cite{shapiro2001duality}, similar to the argument in \cite[Theorem 4.2]{esfahani2015data}.
	 The third equality stems from the fact that the maximum over distributions $F_i$ with respect to the integral is equal to the maximum point of the integrand.
	\end{proof}
	 From Theorem \ref{t:main} our DRSO formulation \eqref{eq:DRO} is equivalent to 
	\begin{equation} \label{eq:DRSO}
			\begin{aligned} 
				&  \inf_{x\in\mathcal{D}} \max_{A \in \Omega(\epsilon_2)} \inf_{u\geq 0} \max_{\xi_i\in\Xi, i \in [n-1]} u\varepsilon_1 + \\
				& \Bigg\{ \frac{1}{n-1}\sum_{i=1}^{n-1} \left[ h(x,Ay_n+\xi_i) - u\|\xi_i - (y_{i+1} - Ay_i)\|^2 \right] \Bigg\}.
			\end{aligned}
	\end{equation}
We can now state our main result.
\begin{thm}\label{t:main1}
 Problems \eqref{eq:DRO} and \eqref{eq:DRSO} are equivalent to the convex-concave minimax problem:
 		\begin{align}
 \inf_{x\in\mathcal{D}}\max_{A,\xi_i\in\Xi, i \in [n-1]} \, &\frac{1}{n-1}\sum_{i=1}^{n-1} h(x,Ay_n+\xi_i) \label{eq:mainProb} \\
 \text{s.t. } &\frac{1}{n-1}\sum_{i=1}^{n-1} \|\xi_i-(y_{i+1}-Ay_i)\| \leq \varepsilon_1, \nonumber \\
 &  \left\| A_{i} - \hat A_{i}\right\| \leq \varepsilon_{2,i} \text{, for $i \in [d]$}. \nonumber
 \end{align} 
\end{thm}
\begin{proof}
		\begin{align}
			&\Psi(x,u,A) \stackrel{.}{=} \max_{\xi_i \in \Xi, i \in [n-1]}  u\varepsilon_1 + \label{eq:equiv}\\
			&\Bigg\{ \frac{1}{n-1}\sum_{i=1}^{n-1} \left[ h(x,Ay_n+\xi_i) - u\|\xi_i - (y_{i+1} - Ay_i)\|^2 \right] \Bigg\} \nonumber
		\end{align}
		is both a maximum of affine functions in $u$ and  a maximum of functions \textit{jointly concave} in $(A,\{\xi_i\})$. Therefore, it is convex in $u$ and concave in $A$. Since the feasible set of $A$ is bounded, by Sion's minimax theorem \cite[Thm.3.4]{sion1958general}, Equation \eqref{eq:DRSO}  becomes
		\begin{align*}
			 &\inf_{x\in\mathcal{D}} \max_{A \in \Omega(\epsilon_2)} \inf_{u\geq 0}\Psi(x,u,A)  \\
			  & \stackrel{\mbox{\cite[Thm.3.4]{sion1958general}}}{=} \inf_{x\in\mathcal{D}}~~\inf_{u\geq 0}\max_{A \in \Omega(\epsilon_2)} \Psi(x,u,A) \\
			  &=  \inf_{x\in\mathcal{D}}~~\inf_{u\geq 0}~~\max_{A \in \Omega(\epsilon_2),\xi_i\in\Xi, i\in[n-1]}\\
			  &u\varepsilon_1+ \frac{1}{n-1}\sum_{i=1}^{n-1} \left[ h(x,Ay_n+\xi_i) - u\|\xi_i - (y_{i+1} - Ay_i)\|^2 \right] \\
			  & \stackrel{\mbox{\cite[Thm.3.4]{sion1958general}}}{=} \inf_{x\in\mathcal{D}}~~\max_{A \in \Omega(\epsilon_2),\xi_i\in\Xi, i\in [n-1] }~~\inf_{u\geq 0}\\
			  &u\varepsilon_1+ \frac{1}{n-1}\sum_{i=1}^{n-1} \left[ h(x,Ay_n+\xi_i) - u\|\xi_i - (y_{i+1} - Ay_i)\|^2 \right]. 
		\end{align*}
		By strong duality applied to the innermost problem, the conclusion follows, after unfolding the definition of $\Omega(\epsilon_2)$ \eqref{eq:Omega}.
\end{proof}
The important consequence of Theorem \ref{t:main1} is that \eqref{eq:DRO} can be solved efficiently by solving the equivalent problem 
\eqref{eq:mainProb}  with techniques such as those in \cite{nemirovski2004prox}.
		
\section{Concentration Inequalities}\label{radius}
We now aim to connect the relaxation parameters $\varepsilon_1$ and $\varepsilon_2$ to the probability of the true probability distribution satisfying the relaxed constraints. We assume that a bound for the support is known, similar to \cite{delage2010distributionally}.
\begin{asm}\label{asm1}
There exists an $R >0$ such that for the noise term $\xi$, we have
$ \|\mathbf{\xi} \|\leq R.$
\end{asm}
We note that the boundedness assumption can be relaxed by requiring square-exponential integrability. This would require techniques for unbounded distributions that involve the Wasserstein distance concentration, as presented in \cite{esfahani2015data}, and consistency of the transition matrix estimation ($A$) (see, e.g.,  \cite{brockwell2016introduction}). For brevity we will focus on the bounded support case only. 
\begin{lem}\label{e_1}
\emph{(Wasserstein metric concentration, specification of $\varepsilon_1$)\\}
Suppose $\xi_1,\xi_2,\cdots,\xi_n \in \mathbb{R}^d$ are i.i.d samples from a distribution $F$  with zero mean and that satisfy Assumption \ref{asm1}. Then, for the empirical distribution $F_n$, the following inequality holds:
\begin{equation} 
\label{eq:wassConc} \P( d_w(F,F_n) \geq \varepsilon) \leq C_0\exp\left(- C_1 N \varepsilon^2\right).
\end{equation}
Here $C_0,C_1$ depend only on $R$ and $d$.
\end{lem}
\begin{proof}
The result is an immediate consequence of \cite[Theorem 3.4]{esfahani2015data}, where we chose $a=2$ and used Assumption \ref{asm1} for bounding $A$ from that statement. 
\end{proof}
We also note that $C_0$, $C_1$ are explicitly computable by using techniques such as in \cite[Appendix B]{gao2016distributionally}. Now, we can select the right-hand side of \eqref{eq:wassConc} to the confidence level, for example, $0.05$. This will be a conservative estimate, however, and we will use cross-validation in practice to compute a suitable $\epsilon_1$, as we will discuss in \S \ref{s:numerical}. The important feature of Lemma \ref{e_1}, however, is the exponential decay of the failure probability with $N$ and $\epsilon^2$. 
\begin{lem}\label{e_2}
\emph{(Concentration of bounded random vectors)\\}
Suppose $\xi_1,\xi_2,\cdots,\xi_n$ are i.i.d samples from a distribution  with zero mean that satisfies  Assumption \ref{asm1}. Then the following holds with probability at least $1-\delta$:
\[ \left\|\sum_{i=1}^n c_i\xi_i \right\| \leq R \sqrt{\sum_{i=1}^nc_i^2}\cdot\left(4\sqrt{d} + 2\sqrt{2\log(1/\delta)}\right). \] 
\end{lem}
\begin{proof}
By Assumption \ref{asm1}, each $\xi_i$ is bounded with zero mean. Applying Hoeffding's lemma \cite[Lemma 1.8]{rigollet201518} to $s^T \xi_i$, we have that  $\xi_i$ are also sub-Gaussian random vectors with variance proxy $R$ \cite[Definition 1.2]{rigollet201518}. In other words,
\[\mathbb{E}e^{\lambda s^T\xi_i} \leq \exp\left(\frac{\lambda^2R^2}{2}\right),\]
for any $\|s\|\leq 1$. From the assumption of independence of $\xi_i$, $i=1,2,\ldots,n-1$, we also have 
\[\mathbb{E}e^{\lambda s^T(\sum_{i=1}^n c_i\xi_i)} \leq \exp\left(\frac{\lambda^2R^2\sum_{i=1}^n c_i^2}{2}\right),\]
for any $\|s\| \leq 1.$ Therefore $\sum_i^n c_i \xi_i$ is a sub-Gaussian random vector with variance proxy $\sigma \doteq R\sqrt{\sum_i^n c_i^2}$.
From \cite[Theorem 1.19]{rigollet201518}, we have with probability greater than $1-\delta$ that 
$\left\|\sum_{i=1}^{n} c_i\xi_i\right\| \leq \sigma(4\sqrt{d} + 2\sqrt{2\log(1/\delta)}).$
\end{proof}


\textbf{Specification of $\varepsilon_{2,i}$}
With this concentration bound, we can now specify the choice of $\varepsilon_{2,i}$. Our model is $Y_+ = A Y_- + E$. 
From the normal equations, we know that the regression matrix estimate is $\hat A = Y_+C$, where $C = Y_-^\dagger$. Here we assume $Y_-Y_-^T$ is invertible. From our model, this results in $\hat A = A + EC$. 
Applying Lemma \ref{e_2} 
to each column of $\hat A$ for confidence level $1-\frac{\delta}{d}$, and using Boole's inequality to the complements, 
we have, with probability greater than $1-\delta=1-d \frac{\delta}{d}$, that $\|A_i - \hat A_i\| \leq \varepsilon_{2,i}$ with 
\[ \varepsilon_{2,i} \leq \sigma_i(4\sqrt{d} + 2\sqrt{2\log(d/\delta)}), \; \sigma_i = R \sqrt{ \sum_{j=1}^{n-1} c_{ji}^2},\; i \in [d].\]
We also note, using regression theory \cite{Vershynin:2010aa}, that the 2-norm of $C$, and thus $\sigma_i$, $i=1,2,\ldots, d$ decays as $O(\frac{1}{\sqrt{n}})$.

\section{Experiments}\label{s:numerical}	
We apply the DRSO approach \eqref{eq:DRO} in the variant outlined in Theorem \ref{t:main1} to a portfolio optimization problem. 
	 The decision variable $x$ is constrained to the $(d-1)$-dimensional standard simplex $\mathcal{D} = \{x \in \mathbb{R}^d | x_1 + \dots+x_d = 1\text{, }x_i\geq 0\text{, }i  =1,\dots,d  \}$. The variable $x$ represents the portions of investment in different stocks. Here the data $y_t \in \mathbb{R}^d$ is the price of $d$ different stocks at time $t$. In the framework of \eqref{eq:SO}, the objective function is the (negative) return
	 	 	\[ h(x,y)= -\langle x,y \rangle. \]
	We subsequently solve the distributionally robust problem \eqref{eq:mainProb} that is derived from our main result, Theorem \ref{t:main1}, with the convex, spherical $\Xi$ from Assumption \ref{asm1}. We report on those results in the rest of this section and label them as "DRO." 

The DRSO problem \eqref{eq:mainProb} was solved with 
 the saddle point algorithm from \cite{nemirovski2004prox} implemented in Julia and 
run on a MacBook Pro, 2.4 GHz Intel Core i5, 8 GB 1600 MHz DDR3.  
The computation time of $100$ experiments for  either synthetic or real data cases below for $n=21$, $d = 8$ (20 time periods) did not exceed $300$ seconds.

\subsection{Synthetic Data}

	For our experiment with synthetic data, the feasible set $\mathcal{D}$ is the $(d-1)$-dimensional standard simplex, and we set $d = 8$. The objective function is the inner product $-\langle x,y\rangle$. Here $y_i$ is from the VAR(1) times series, with the transition matrix entrywise drawn from uniform distribution over $[0,1]$, then scaled so that $\|A\| = 0.8$ and $\xi_t$ is from $N(0,R^2I)$, then truncated to $2$-norm no greater than a preset radius $R$.	The metric we use in the Wasserstein distance constraint is the $2$-norm in Euclidian space $(\mathbb{R}^d, \|\cdot \|)$. The radii of confidence intervals from \S \ref{radius} are conservative. 
	For better performance we shrink the parameters $\epsilon_1, \epsilon_2$ by factors $1, 0.5$ on the first $40$ data points and choose the combination with the best outcome. In \cite{esfahani2015data}, the authors tried different confidence levels $\delta$, which fundamentally resulted in the same effect.
	We compare the solution of DRO $x^d$ and the solution of sample average approximation (SAA) $x^s$ with the empirical residuals. 
	Here our SAA solution is obtained by 
	 \begin{equation} \label{saa_formulation}
		\begin{aligned}
			\min_{x} \max_{A} &\sum_{i=1}^{n-1}h(x,Ay_n +\hat{\xi}_i )\\
			\text{s.t.}& \|A_i-\hat A_i\| \leq \varepsilon_{2,i}\text{, }i\in[d].
		\end{aligned}
	\end{equation} 
	We also calculate the solution of the deterministic version of \eqref{eq:DRO} obtained by plugging in the maximum likelihood estimator (MLE) of $y_{n}$, $\hat{A} y_{n-1}$.
	Let $x^*=\arg\min h(x,y_{n+1})$  (solution with perfect information). We report the "regret" $h(x,y_{n+1})-h(x^*,y_{n+1})$ for $x$ given by the different approaches. Some empirical quantiles are given in Table \ref{fig:linear_exp}, and two histograms are given in Figures \ref{fig:linear_1} and \ref{fig:linear_2}. As we can see from the results, with more training samples or lower noise level, the estimated transition matrix $\hat A$ becomes more accurate, so regression results in a decision closer to the perfect one most of the time. In all scenarios, however, DRSO has a lighter tail (see also Figures \ref{fig:linear_1} and \ref{fig:linear_2}), which demonstrates the robustness of our decision. In particular, in Table \ref{fig:linear_exp}, DRSO exhibits the smallest regret for all experiments at the 90th quantile. 

\begin{figure}
	\centering\includegraphics[width = 2.8in]{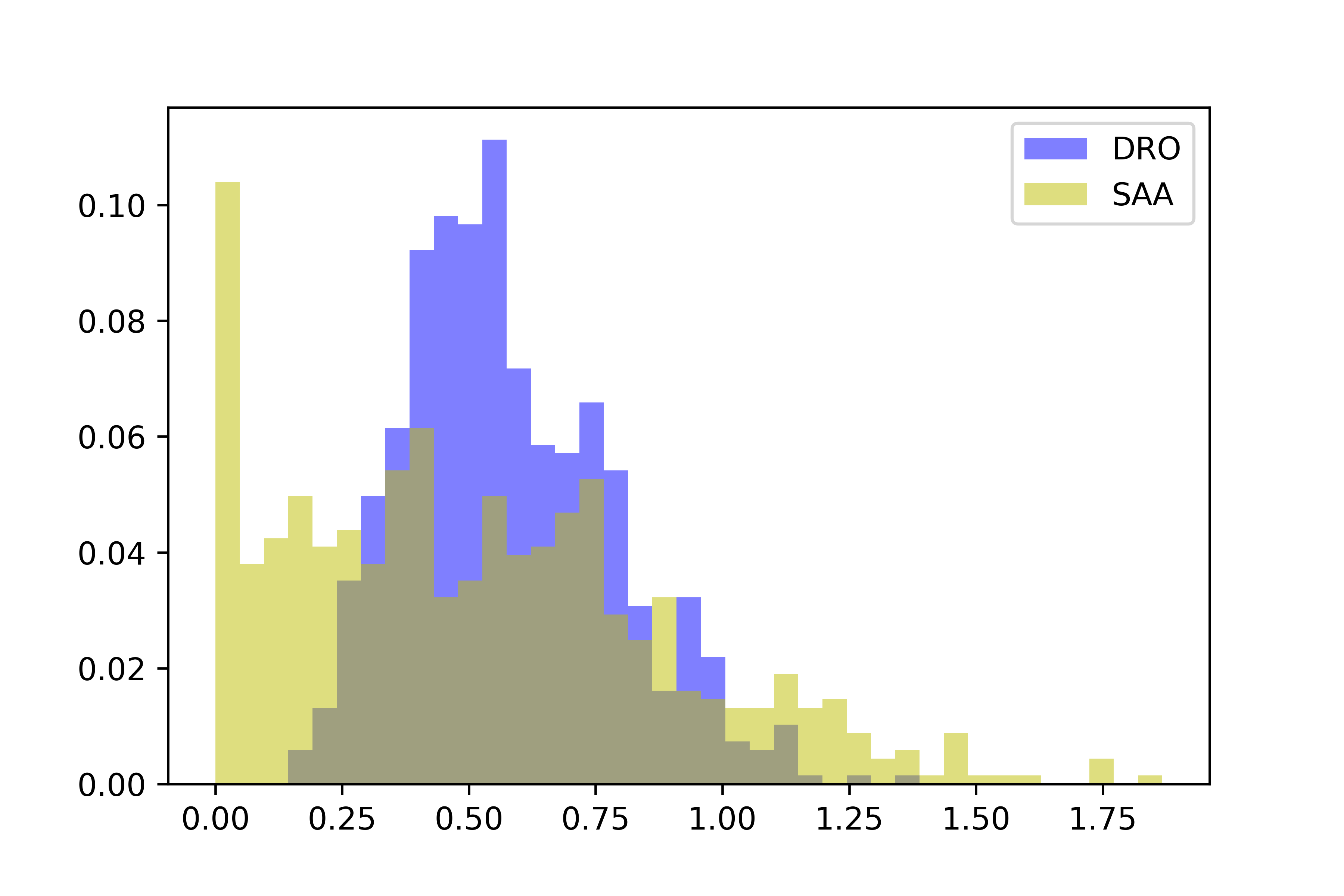}
	\caption{Comparison of DRO and SAA, synthetic data. All data are normalized by noise radius $R$. Noise radius is $16$. Sample size is $8$. Problem dimension is $5$.}
	\label{fig:linear_1}
\end{figure}

\begin{figure}

	\centering\includegraphics[width = 2.8in]{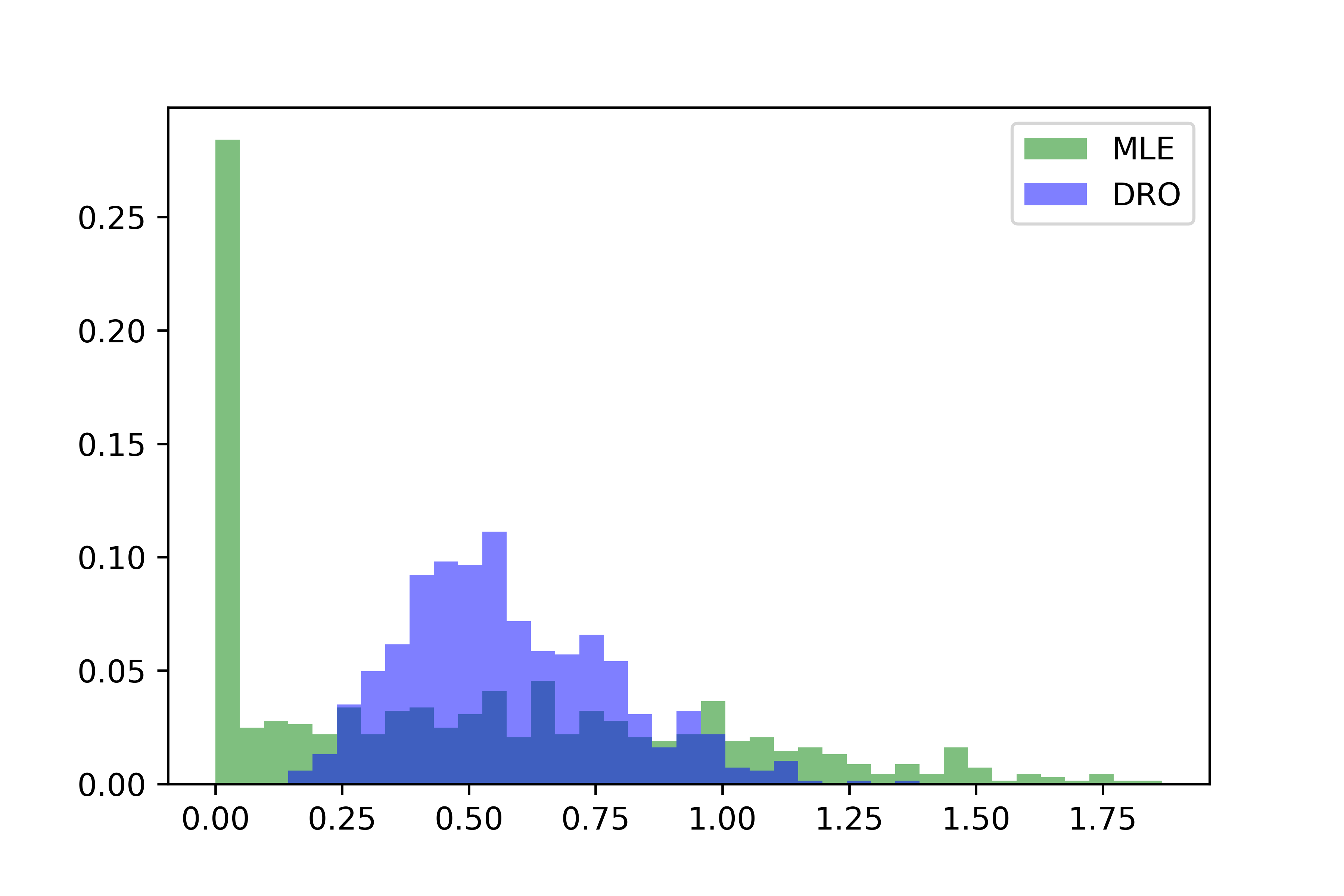}
	\caption{Comparison of DRO and MLE, synthetic data. All data are normalized by noise radius $R$. Noise radius is $16$. Sample size is $8$. Problem dimension is $5$.}
	\label{fig:linear_2}
\end{figure}


	\begin{table}
		\scriptsize
		\begin{tabular}{||c c c c ||}
			\hline
			Setup& Median & 75th Perc. & 90th Perc. \\
			\hline\hline		
8-4-5	&	       		\textbf{0.49}/0.59/0.52&	0.91/\textbf{0.74}/0.78 &	1.27/\textbf{0.92}/1.09\\	  

8-16-5&			\textbf{0.42}/0.55/0.47&	0.81/\textbf{0.72}/0.75 &		1.14/\textbf{0.86}/1.05 \\	

16-4-5&			\textbf{0.34}/0.57/0.41&	0.72/0.72/\textbf{0.64} &		1.08/\textbf{0.86}/\textbf{0.86} \\	

16-16-5&			\textbf{0.38}/0.57/0.43&	0.76/0.70/\textbf{0.68} &		1.03/\textbf{0.82}/0.92 \\	

16-4-8&			\textbf{0.43}/0.53/0.48&	0.75/\textbf{0.66}/\textbf{0.66} &		0.99/\textbf{0.76}/0.80 \\	

16-16-8&			\textbf{0.43}/0.54/0.46&	0.73/\textbf{0.65}/0.68 &		1.00/\textbf{0.77}/0.83 \\	
			\hline
		\end{tabular}
		\caption{Comparison of MLE, SAA, and DRO for several standard percentiles, synthetic data. Setup is "sample size-($n-1$)-noise radius-($R$)-dimension($d$)." Different statistics are normalized by noise radius $R$ and given by "MLE/DRO/SAA." Lowest regret among methods is boldfaced.}
		\label{fig:linear_exp}
	\end{table}

\subsection{Real Data}
	We perform our real data analysis with the asset price of nine tech companies from the S\&P 500 (INTC, AMZN, FB, MSFT, GOOGL, IBM, ORCL, ADBE, AAPL) from January  2013 to January 2018 with our model on the log price $y_t = \log(p_{i})$ at the end of each day and the objective function the approximated return $\sum_{i=1}^d x_i (\frac{p_{i+1}}{p_i}-1) \approx \langle x, y_T - y_{T-1}\rangle$ (where one uses the approximation $e^r\approx 1+r$, which is very accurate in the range of successive daily price ratios)~\cite{ruppert2011statistics}. We again compare the DRO and SAA models. In addition, we  run the algorithm assuming independence between the samples (which we call "Independent DRO"). For each day, the algorithms are allowed to use data from the previous $15$ days. The parameters are chosen by experimenting on the first three months of the dataset with, in reference to Theorem \ref{t:main1} and \S \ref{radius}, $\delta = 0.05, 0.1$, and $R =  1\%, 4\% , 10\%$ and selecting the combination with the best accumulated return. 
	The results are shown in Table \ref{table:real} for some quantiles of actual returns if we invest $10,000$ dollars each day, $10000 \sum_{i=1}^d x_i (\frac{p_{i+1}}{p_i}-1) $, where $x$ is, in turn, the solution for the four approaches. We see that both robust methods have a significantly lighter tail than does either the SAA or MLE approach, that our AR-based DRO on performs better than the independent DRO (except only slightly for the median), and that ignoring the uncertainty results in a significant degradation (MLE). 

\begin{table}
	\footnotesize
	\begin{tabular}{|| c c c c c||}
	\hline
		&Mean & Median & 25th Perc. & 10th Perc. \\
		\hline\hline
	DRO&0.822			&0.987			&\textbf{-3.888}		&\textbf{-9.758}\\
	SAA&0.955			&\textbf{1.032}		&-3.896		&-10.013\\
	MLE& \textbf{1.433}		&0.733			&-6.032			&-13.704\\
	Independent &0.819		&0.988			&-3.889			&-9.891\\
	DRO &&&&\\	
	\hline
	\end{tabular}
	\caption{Comparison of statistics of daily return for real stock data. \label{table:real} 
	}
\end{table}
\section*{Acknowledgment}
This material was based upon work
supported by the U.S. Department of Energy, Office of Science,
Office of Advanced Scientific Computing Research (ASCR) under
Contract DE-AC02-06CH11347 and by NSF
through award CNS-1545046. An initial version of this work was issued as 
 Preprint ANL/MCS-P9163-0219, Argonne National Laboratory. 




\bibliographystyle{model1-num-names}
\bibliography{sample.bib}






\vspace{-0.15cm}
\begin{flushright}
	\scriptsize \framebox{\parbox{2.5in}{Government License: The
			submitted manuscript has been created by UChicago Argonne,
			LLC, Operator of Argonne National Laboratory (``Argonne").
			Argonne, a U.S. Department of Energy Office of Science
			laboratory, is operated under Contract
			No. DE-AC02-06CH11357.  The U.S. Government retains for
			itself, and others acting on its behalf, a paid-up
			nonexclusive, irrevocable worldwide license in said
			article to reproduce, prepare derivative works, distribute
			copies to the public, and perform publicly and display
			publicly, by or on behalf of the Government. The Department of Energy will provide public access to these results of federally sponsored research in accordance with the DOE Public Access Plan. http://energy.gov/downloads/doe-public-access-plan. }}
	\normalsize
\end{flushright}
\end{document}